\documentclass[12pt,draft,reqno]{amsart}

\usepackage{amsmath,amssymb,amscd,amsfonts,amsthm}
\usepackage{graphics}
\usepackage{dsfont}
\usepackage[all]{xy}
\usepackage{enumerate}
\usepackage{dsfont}
\usepackage{bbold}
\usepackage{mathbbol}
\usepackage{mathrsfs}

\usepackage{array,amscd,latexsym, verbatim}

{\catcode`\@=11
\gdef\n@te#1#2{\leavevmode\vadjust{%
 {\setbox\z@\hbox to\z@{\strut#1}%
  \setbox\z@\hbox{\raise\dp\strutbox\box\z@}\ht\z@=\z@\dp\z@=\z@%
  #2\box\z@}}}
\gdef\leftnote#1{\n@te{\hss#1\quad}{}}
\gdef\rightnote#1{\n@te{\quad\kern-\leftskip#1\hss}{\moveright\hsize}}
\gdef\?{\FN@\qumark}
\gdef\qumark{\ifx\next"\DN@"##1"{\leftnote{\rm##1}}\else
 \DN@{\leftnote{\rm??}}\fi{\rm??}\next@}}
\DeclareFontFamily{OT1}{wncyr}{\hyphenchar\font45 }
\DeclareFontShape{OT1}{wncyr}{m}{n}{%
   <5> <6> <7> <8> <9> gen * wncyr
   <10> <10.95> <12> <14.4> <17.28> <20.74>  <24.88>wncyr10}{}
\DeclareFontShape{OT1}{wncyr}{m}{it}{%
   <5> <6> <7> <8> <9> gen * wncyi
   <10> <10.95> <12> <14.4> <17.28> <20.74> <24.88> wncyi10}{}
\DeclareFontShape{OT1}{wncyr}{m}{sc}{%
   <5> <6> <7> <8> <9> <10> <10.95> <12> <14.4>
   <17.28> <20.74> <24.88>wncysc10}{}
\DeclareFontShape{OT1}{wncyr}{b}{n}{%
   <5> <6> <7> <8> <9> gen * wncyb
   <10> <10.95> <12> <14.4> <17.28> <20.74> <24.88>wncyb10}{}
\input cyracc.def

\DeclareMathSizes{9}{9}{7}{5} 

\theoremstyle{plain}

\newtheorem{theorem}{Theorem}

\newtheorem{lemma}{Lemma}
\newtheorem{remark}{\bf Remark}

\theoremstyle{definition}

\newtheorem{nothing*}[theorem]{}
\newtheorem{subnothing*}[sub]{}

\theoremstyle{remark}

\def\bA1{{\mathbf A}\!^1}
\def\P1{{\bf P}^1}

 \newcommand{\bwedge}{\mbox{\large{$\wedge$}}}

\begin{document}

\title[Number of components of the nullcone]
{Number of components of\\ the nullcone}

\author[Vladimir L. Popov]{Vladimir L. Popov${}^*$}
\address{Steklov Mathematical Institute,
Russian Academy of Sciences, Gub\-kina 8, Moscow\\
119991, Russia}

\email{popovvl@mi.ras.ru}

\thanks{${}^*$\,This work is supported by the RSF under a grant 14-50-00005.}

\begin{abstract} For every pair $(G, V)$ where $G$ is a connected simple linear algebraic group and  $V$ is a simple algebraic $G$-module with a free algebra of invariants, the number of irreducible components of the nullcone of unstable vectors in $V$ is found.
\end{abstract}

\maketitle

{\bf 1.} We fix as the base field an algebraically closed field $k$ of characteristic zero.\;Below the standard notation and terminology of the theory of algebraic groups and invariant theory  \cite{PV} are used freely.

Consider a finite dimensional vector space  $V$ over the field $k$ and a connected semisimple algebraic subgroup $G$ of the group ${\rm GL}(V)$.
 Let $\pi_{G, V}\colon V\to V/\!\!/G$ be the categorical quotient for the action of  $G$ on $V$, i.e.,  $V/\!\!/G$ is the irreducible affine algebraic variety with the coordinate algebra
$k[V]^G$ and the morphism $\pi_{G, V}^{\ }$ is determined by the identity embedding $k[V]^G \hookrightarrow k[V]$.\;Denote by
${\mathcal N}_{G, V}$ the nullcone of the $G$-module $V$, i.e., the fiber $\pi_{G, V}^{-1}(\pi_{G, V}(0))$
of the morphism $\pi_{G, V}^{\ }$.\;A point of the space $V$ lies in ${\mathcal N}_{G, V}$
if and only if its $G$-orbit is nilpotent, i.e., contains in its closure the zero of the space $V$ (see\;\cite[5.1]{PV}).

This article owes its origin to the following A.\;Joseph's question
 \cite{J}: may it happen that the nullcone ${\mathcal N}_{G, V}$ is reducible if the group $G$ is simple,
 its natural action on $V$ is irreducible, and the algebra of invariants $k[V]^G$ is free?

 Pairs $(G, V)$ with a free algebra of invariants  $k[V]^G$
 have been stu\-di\-ed intensively in the $70$s of the last century
  (see\;\cite{PV}, \cite{P5} and the literature cited there). Under the assumptions of simplicity of the group $G$ and irreducibility of its action on $V$ they are completely classified and constitute a remarkable class which admits
  a number of other important characterizations.

In Theorem \ref{main} proved below we find the number of irreducible compo\-nents of the nullcone
${\mathcal N}_{G, V}$ for every pair $(G, V)$ from this class. As a corollary we obtain the affirmative answer to
A.\;Joseph's question. The proof is based on the aforementioned classification and characterizati\-ons
that are reproduced below in Theorems
\ref{list} and\;\ref{prop}.

\vskip 2mm

{\bf 2.} Up to conjugacy in ${\rm GL}(V)$, the group $G$ is uniquely determined as the image of a representation $\widetilde G\to {\rm GL}(V)$
of its universal covering group $\widetilde G$.\;The equivalence class on this representation, if it is irreducible, is  uniquely determined by its highest weight $\lambda$ (with respect to a fixed maximal torus and a Borel subgroup of the group $\widetilde G$ containing this torus). With this in mind, we shall write
$G=({\sf R}, \lambda)$, where ${\sf R}$ is the type of the root system of the group
$G$. Note that  $({\sf R}, \lambda)=({\sf R}, \lambda^*)$, where $\lambda^*$ is the highest weight of the dual representation. We denote by $\varpi_1, \ldots, \varpi_r$ the fundamental weights of the group
$\widetilde G$ numbered as in Bourbaki
 \cite{B}.\;If ${\sf R}={\sf A}_r, {\sf B}_r, {\sf C}_r, {\sf D}_r$, then we assume that, respectively, $r\geqslant 1, 3, 2, 4$.

The following theorem is proved in \cite{KPV}:

\begin{theorem} \label{list}
All connected nontrivial simple algebraic subgroups $G$ of the group ${\rm GL}(V)$
that act on $V$ irreducibly and have a free algebra of invariants $k[V]^G$, are exhausted by the following list:
\begin{enumerate}[\hskip 4.2mm\rm(i)]
\item {\rm(}adjoint groups{\rm):}
\begin{equation*}
\begin{gathered}
({\sf A}_r, \varpi_1+\varpi_r);  ({\sf B}_r, \varpi_2); ({\sf D}_r,\varpi_2); ({\sf C}_r, 2\varpi_1);\\
({\sf E}_6, \varpi_2), ({\sf E}_7, \varpi_1); ({\sf E}_8, \varpi_8); ({\sf F}_4, \varpi_1); ({\sf G}_2, \varpi_2) \end{gathered}
\end{equation*}
\item {\rm(}isotropy groups of symmetric spaces{\rm):}
\begin{equation*}
\begin{gathered}
({\sf B}_r, \varpi_1); ({\sf D}_r, \varpi_1); ({\sf A}_3, \varpi_2);
({\sf A}_1, 2\varpi_1);\\
({\sf B}_r, 2\varpi_1); ({\sf D}_r, 2\varpi_1); ({\sf A}_3, 2\varpi_2); ({\sf C}_2, 2\varpi_1); ({\sf A}_1, 4\varpi_1);\\
({\sf C}_r, \varpi_2); ({\sf A}_7, \varpi_4); ({\sf B}_4, \varpi_4); ({\sf C}_4, \varpi_4);
({\sf D}_8, \varpi_8); ({\sf F}_4, \varpi_4);
\end{gathered}
\end{equation*}
\item {\rm(}groups $G$ with $k[V]^G=k${\rm):}
\begin{equation*}
\begin{gathered}
({\sf A}_r, \varpi_1); ({\sf A}_r, \varpi_2), r\geqslant 4 \mbox{ \it even\,}; ({\sf C}_r, \varpi_1);
({\sf D}_5, \varpi_5);
\end{gathered}
\end{equation*}
\item {\rm(}groups $G$ with $Õ{\rm tr\,deg}k[V]^G=1$ not included in {\rm(i)} and {\rm (ii)}{\rm):}
\begin{equation*}
\begin{gathered}
({\sf A}_r, 2\varpi_1), r\geqslant 2;  ({\sf A}_r, \varpi_2), r\geqslant 5 \mbox{ \it odd\,};\\
({\sf A}_1, 3\varpi_1); ({\sf A}_5, \varpi_3); ({\sf A}_6, \varpi_3);  ({\sf A}_7, \varpi_3);\\
({\sf B}_3, \varpi_3); ({\sf B}_5, \varpi_5); ({\sf C}_3, \varpi_3); ({\sf D}_6, \varpi_6); ({\sf D}_7, \varpi_7);\\ ({\sf G}_2, \varpi_1); ({\sf E}_6, \varpi_1);
({\sf E}_7, \varpi_7);
\end{gathered}
\end{equation*}
\item {\rm(}other groups{\rm):}
\begin{equation*}
\begin{gathered}
({\sf A}_2, 3\varpi_1); ({\sf A}_8,\varpi_3); ({\sf B}_6, \varpi_6).
\end{gathered}
\end{equation*}
\end{enumerate}
\end{theorem}

\begin{remark}\label{rem1} {\rm There are no repeated groups inside each of these five lists (i)--(v). The unique group included in two different lists  (namely, in (i) and (ii)) is $({\sf A}_1, 2\varpi_1)$.
The groups $G$
with $Õ{\rm tr\,deg}k[V]^G=1$ included in at least one of the lists (i), (ii) are
$({\sf B}_r, \varpi_1)$, $({\sf D}_r, \varpi_1)$, $({\sf A}_3, \varpi_2)$,
$({\sf C}_2, \varpi_2)$,
$({\sf A}_1, 2\varpi_1)$, $({\sf B}_4, \varpi_4)$
and only these groups.}
\end{remark}

{\bf 3.} Recall from \cite[3.8, 8.8]{PV}, \cite[Chap.\,5, \S\,1, 11]{P5}, \cite{P4}
that an algeb\-ra\-ic subvariety
$S$ in $V$ is called  {\it a Chevalley section with the Weyl group} $W(S):=N(S)/Z(S)$, where $N(S):=\{g\in G\mid g\cdot S=S\}$ and $Z(S):=\{g\in G\mid g\cdot s=s \;\forall s\in S\}$, if the homomorphism of
$k$-algebras
$k[V]^G\to k[S]^{W(S)}, f\mapsto f|_S,$ is an isomorphism.\;A linear subvariety in
$V$ that is a Chevalley section with trivial Weyl group
(i.e., a linear subvariety intersecting every fiber of the morphism
 $\pi_{G, V}^{\ }$ at a single point) is called
{\it a Weierstrass section}. A linear subspace in $V$ that is a Chevalley section with a finite Weyl group is called {\it a Cartan subspace}.

Recall  also (see \cite[Thm.\;3.3 and Cor.\;4 of Thm.\;2.3]{PV}) that semisimp\-li\-ci\-ty of the group $G$ implies the equality
\begin{equation}\label{max}
m_{G, V}^{\ }:=\max_{v\in V} \dim G\cdot v=\dim V-\dim V/\!\!/G.
\end{equation}

Consider the following properties:

\begin{enumerate}[\hskip 6.2mm\rm(i)]
\item[(FA)] $k[V]^G$ is a free $k$-algebra;
\item[(FM)] $k[V]$ is a free $k[V]^G$-module;
\item[(ED)] all fibers of the morphism $\pi_{G, V}^{\ }$ have the same dimension;
\item[(${\rm ED}_0$)]
$\dim {\mathcal N}_{G, V}=m_{G, V}^{\ }$
(see\;\eqref{max});
\item[(FO)] every fiber of the morphism $\pi_{G, V}^{\ }$ contains only finitely many $G$-orbits;
\item[$({\rm FO}_0)$] ${\mathcal N}_{G, V}$ contains only finitely many $G$-orbits;
\item[(NS)] $G$-stabilizers of points in general position in $V$ are nontrivial;
    \item[(CS)] there is a Cartan subspace in $V$;
        \item[(WS)] there is a Weierstrass section in $V$.
\end{enumerate}
The following implications between them hold true:
\begin{align*}
{\rm (FM)}&\Leftrightarrow {\rm (FA)}\&{\rm(ED)}\;\;\mbox{(see\,\cite[p.\,127, Thm.\,1]{P5});}\\
({\rm ED}_0)&\Leftrightarrow {\rm  (ED)}\Leftarrow ({\rm FO}_0)\;\;\mbox{(see\,\cite[p.\,128, Thm.\,3, Cor.]{P5});}\\
({\rm FO}_0)&\Leftrightarrow {\rm (FO)}\;\;\mbox{(see\,\cite[Cor.\,3 of Prop.\,5.1]{PV});}\\
{\rm (CS)} &\Rightarrow {\rm (FM)}\Leftarrow {\rm (WS)}\;\;\mbox{(see\,\cite[p.\,133, Thm.\,7]{P5})}.
\end{align*}

\begin{theorem}\label{prop} For the connected simple algebraic subgroups $G$ in ${\rm GL}(V)$, acting on $V$ irreducibly, all nine properties $\rm (FA)$, $\rm (FM)$, $\rm (ED)$,
$({\rm ED}_0)$,
$\rm (FO)$, $({\rm FO}_0)$, ${\rm (NS)}$, $\rm (CS)$, and $\rm (WS)$
are equivalent\footnote{In \cite[p.\;207, Thm.]{MFK}, the property
    (NS) is replaced by the property that the $G$-stabilizer of {\it every} point of $V$ is nontrivial.\;It is a mistake: for instance, the
    ${\rm SL}_2$-module of binary forms in $x$ and $y$ of degree $3$ has the property (FA), but the ${\rm SL}_2$-stabilizer of the form $x^2y$ is trivial.}.
\end{theorem}

\begin{proof} The complete list of the groups $G$ having the pro\-per\-ty (FA) is obtained in \cite{KPV}; the one having the property (ED) is obtained in \cite{P1}, \cite[p.\;141, Thm.\,8]{P5} and, in the same papers, that having the property (FM);
the one having the property (FO) is obtained in \cite{K}.
The results of papers \cite{AVE}, \cite{AP}, \cite{AP1}, \cite{AP2}
yield the complete list of the groups $G$ having the property (NS). Matching the obtained lists proves the equivalence of the properties  (FA), (FM), (ED), (FO), and
(NS) (see\;\cite[Thm.\;8.8]{PV} and \cite[p.\;127, Thm.\,1]{P5}).
It is proved in \cite[p.\;142, Thm.\,9]{P5} that each of the properties   (CS) and (WS) is equivalent to the property (ED).
\end{proof}

\begin{remark} {\rm The conditions of simplicity of the group $G$ and irreduci\-bi\-li\-ty of its action on $V$
in Theorem \ref{prop} are essential, see\;\cite{P4}.}
\end{remark}

{\bf 4.} Now we turn to finding the number of irreducible components of the nullcone ${\mathcal N}_{G, V}$.

\begin{lemma}\label{<2} If $\dim V/\!\!/G\leqslant 1$, then the nullcone ${\mathcal N}_{G, V}$ is irreducible.
If $\dim V/\!\!/G=0$, then it contains an open dense $G$-orbit.
\end{lemma}

\begin{proof}
The equality $\dim V/\!\!/G=0$
 means that $\dim V/\!\!/G$ is a single point.\;By the definition of the nullcone, the latter condition is equivalent to the equality   ${\mathcal N}_{G, V}=V$. In particular, in this case the nullcone ${\mathcal N}_{G, V}$ is irreducible.\;On the other hand, in view of  \eqref{max}, the equality $\dim V/\!\!/G=0$
is equivalent
to that  $V$ contains a $G$-orbit of dimension $\dim V$, i.e., an open and dense orbit.

In view of smoothness of  $V$, the algebraic variety $V/\!\!/G$ is normal (see.\,\cite[Thm.\;3.16]{PV}).
Let $\dim\,V/\!\!/G=1$. It follows from rationality of the algebraic variety $V$,
dominance of the morphism $\pi_{G, V}^{\ }$, and L\"uroth's theorem that the curve
$V/\!\!/G$ is rational.\;Being normal, it is smooth. Hence  $V/\!\!/G$ is isomorphic to an open subset of the affine line.\;Since every invertible element of the algebra $k[V]$ is a constant,
the algebra  $k[V]^G$ has the same property. Hence the curve $V/\!\!/G$ is isomorphic to the affine line,
and therefore, there is a polynomial $f\in k[V]^G$ such that $f(0)=0$ and $k[V]^G=k[f]$.\;Since the group $G$ is connected and has no nontrivial characters, the polynomial $f$ is irreducible (see\;\cite[Thm.\;3.17]{PV}).\;Since
${\mathcal N}_{G, V}=\{v\in V\mid f(v)=0\}$, this implies irreducibili\-ty of the nullcone
${\mathcal N}_{G, V}$.
\end{proof}

\begin{theorem}\label{main} The nullcone ${\mathcal N}_{G, V}$ of the connected nontrivial simple algebraic group  $G\subseteq {\rm GL}(V)$ acting irreducibly on $V$ and having the equivalent properties listed in Theorem {\rm \ref{prop}} is reducible if and only if  $G$ is contained in the following list:
\begin{equation}\label{exc}
({\sf D}_r, 2\varpi_1), ({\sf A}_3, 2\varpi_2), ({\sf A}_7, \varpi_4).
\end{equation}
For every group $G$ from list  {\rm \eqref{exc}}, the number of irreducible components of the nullcone ${\mathcal N}_{G,V}$ is equal to $2$.
\end{theorem}
\begin{proof}
From Theorem \ref{prop} we obtain the following interpre\-ta\-ti\-on of the number of irreducible components of the nullcone ${\mathcal N}_{G,V}$. Using
 \eqref{max} and  the fiber dimension theorem
  (see\;\cite[Chap.\;II, \S 3]{H}), we infer that dimension of every irreducible component of the nullcone
   ${\mathcal N}_{G,V}$ is at least $m_{G, V}^{\ }$.\;This and the property
$({\rm ED}_0)$ imply that dimension of every irreducible component of the nullcone ${\mathcal N}_{G,V}$
is equal to $m_{G, V}^{\ }$. But in view of the property $({\rm FO}_0)$ every irreducible component of the nullcone ${\mathcal N}_{G,V}$ is the closure of some  $G$-orbit.\;Hence the number of irreducible components of the nullcone ${\mathcal N}_{G,V}$ is equal to the number of
$m_{G, V}^{\ }$-dimensional nilpotent $G$-orbits in $V$.

Now we shall use Theorem \ref{list} and find, for every group $G$ listed in it, the number of irreducible components of the nullcone ${\mathcal N}_{G,V}$.

\vskip 1mm

1. If the group $G$ is adjoint, then according to \cite[Cor.\;5.5]{K}, the nullcone ${\mathcal N}_{G,V}$ is irreducible.\;This conclusion covers all the groups $G$ from list (i) of Theorem\;\ref{list}.

\vskip 1mm

2. In view of Lemma \ref{<2}, the nullcone ${\mathcal N}_{G,V}$ is irreducible
for all the groups $G$ from lists (iii) and (iv) of Theorem \ref{list} and also for the groups with
${\rm tr\hskip .25mm deg}_kk[V]^G=1$ mentioned in Remark \ref{rem1}.

\vskip 1mm

3. Consider all the groups $G$ from list (v) of Theorem\;\ref{list}.

\vskip 1mm
(3a) The orbits of the group $({\sf A}_2, 3\varpi_1)$ are the orbits of the natural action of the group
${\rm SL}_3$ on the space of cubic forms in three variables. According to \cite[5.4, Example $2^\circ$]{PV},
the Hilbert--Mumford criterion imp\-lies the existence of a linear subspace $L$ in $V$ such that
 ${\mathcal N}_{G,V}=G\cdot L$. Hence the nullcone
 ${\mathcal N}_{G,V}$ is irreducible.

\vskip 1mm
(3b)  The orbits of the group $({\sf A}_8, \varpi_3)$
 are the orbits of the natural action of the group
 ${\rm SL}_9$ on the space of $3$-vectors
  $\bwedge^3 k^9$. The classifi\-ca\-ti\-on of them is obtained in \cite{VE}; it shows  (see\;\cite[Table\;6, $\dim S=0$]{VE}) that in this case there is a unique nilpotent orbit of dimension   $m_{G, V}^{\ }=80$. Hence the nullcone
${\mathcal N}_{G,V}$ is irreducible.

\vskip 1mm
(3c) The orbits of the group $({\sf B}_6, \varpi_6)$
 are the orbits of the natural action of the group
 ${\rm Spin}_{13}$ on the space of spinor
 representation. The classifi\-ca\-ti\-on of them is obtained in  \cite{KV}; it shows (see\;\cite[Thm.\;1(3)]{KV})
  that in this case there is a unique nilpotent orbit of dimension
  $m_{G, V}^{\ }=62$, and hence the nullcone
${\mathcal N}_{G,V}$ is irreducible.

\vskip 1mm

4. Let us now consider all the groups  $G$ from the remaining list (ii) of Theorem\;\ref{list}. By virtue of the Lefschetz principle, we may (and shall) assume that
 $k=\mathbb C$.\;All these groups are obtained by means of the following general construction.

Consider a semisimple complex Lie algebra $\mathfrak h$, it adjoint group  ${\rm Ad}\,\mathfrak h$, and an involution
$\theta\in {\rm Aut}\,\mathfrak h$. The decomposition
\begin{equation*}
\mathfrak h=\mathfrak k\oplus \mathfrak p, \;\;\mbox{ where }\;\mathfrak k:=\{x\in \mathfrak h\mid \theta(x)=x\},\; \mathfrak p:=\{x\in \mathfrak h\mid \theta(x)=-x\}.
\end{equation*}
is a  $\mathbb Z_2$-grading of the Lie algebra  $\mathfrak h$, and $\mathfrak k$ is its proper reductive subalgebra (see\;\cite{VO}).\;Let $K$ be the connected algebraic subgroup of
${\rm Ad}\,\mathfrak h$
with the Lie algebra $\mathfrak k$.\;The subspace $\mathfrak p$ is invariant with respect to the restriction to $K$ of the natural action of the group ${\rm Ad}\,\mathfrak h$ on $\mathfrak h$.
The action of  $K$ on $\mathfrak p$ arising this way determines a homomorphism
$\iota\colon K\to {\rm GL}(\mathfrak p).
$

For every group from list  (ii) of Theorem\;\ref{list}, there is a pair $({\mathfrak h, \theta})$ such that
$V=\mathfrak p$ and $G=\iota(K)$.

Next, we use the following facts (see \cite{KR}, \cite{CM}, \cite{VO}, \cite{PT}).

 In $\mathfrak h$, there is a $\theta$-stable real form $\mathfrak r$ of the Lie algebra $\mathfrak h$, such that
$
\mathfrak r= (\mathfrak r\cap \mathfrak k)
\oplus (\mathfrak r\cap \mathfrak p)$ is its Cartan
decomposition
 (thereby $\mathfrak r\cap \mathfrak k$ is a compact real form of the Lie algebra $\mathfrak k$).\;The semisimple real Lie algebra $\mathfrak r$ is noncompact and
the juxtaposition $\mathfrak r \rightsquigarrow \theta$
determines a bijections between
the noncompact real forms of the Lie algebra
 $\mathfrak h$, considered up to an isomorphism, and the involutions in
${\rm Aut}\,\mathfrak h$, considered up to conjugation.
By means of this bijection and described construction,
every group
$G$ from list (ii) of Theorem\;\ref{list} is determined by
some noncompact semi\-simp\-le real Lie algebra
 $\mathfrak s$; we say that $G$ and $\mathfrak s$ {\it correspond} each other.

The nullcone
${\mathcal N}_{K,\mathfrak p}$ for the action of  $K$ on $\mathfrak p$ contains only finitely many  $K$-orbits,
therefore,  every its irreducible component contains an open dense  $K$-orbit;
the latter is called  {\it principal} nilpotent $K$-orbit and its dimension is equal to the maximum of dimensions of
$K$-orbits in $\mathfrak p$.

Let
$\sigma\colon \mathfrak h\to \mathfrak h,\;x+iy\mapsto x-iy,\;x, y\in {\mathfrak r}$.\;Denote by ${\mathcal N}_{\mathfrak r}$ the set of nilpotent elements of the Lie algebra
$\mathfrak r$.\;In every nonzero $K$-orbit $\mathscr O\subset{\mathcal N}_{K,\mathfrak p}$,
 there is an element $e$ such that $\{e, f:=-\sigma(e), h:=[e, f]\}$ is an ${\mathfrak sl}_2$-triple (i.e., $[h, e]=2e$ and $[h, f]=-2f$).\;Then the element $(i/2)(e+f-h)$ lies in ${\mathcal N}_{\mathfrak r}$, its ${\rm Ad}\,\mathfrak r$-orbit ${\mathscr O}'$
does not depend on the choice of  $e$, the equality
$2\dim_{\mathbb C} \mathscr O=\dim_{\mathbb R}\mathscr O'$ holds, and the map
 $\mathscr O\mapsto \mathscr O'$ is a bijection between the set
 of nonzero
 $K$-orbits in ${\mathcal N}_{K,\mathfrak p}$ and the set of nonzero
 ${\rm Ad}\,\mathfrak r$-orbits in
${\mathcal N}_{\mathfrak r}$.

A nilpotent element of a real semisimple Lie algebra
$\mathfrak s$
is called {\it compact} if the reductive Levi factor of its centralizer in
 $\mathfrak s$ is a compact Lie algebra, \cite{PT}.\;For all simple real
 Lie algebras $\mathfrak s$ and their compact elements  $x$, the orbits $({\rm Ad}\,\mathfrak s)\cdot x$
are classified  (and their dimensions are found) in \cite{PT}. If, in the above notation,
the elements of an  ${\rm Ad}\,\mathfrak r$-orbit $\mathscr O'$ are compact, then the $K$-orbit $\mathscr O$ is called $(-1)$-{\it distinguished}, \cite{P6}. All principal nilpotent
$K$-orbits are $(-1)$-distinguished, \cite{P3}.

It follows from the aforesaid
that the number of irreducible compo\-nents of the nullcone
 ${\mathcal N}_{K,\mathfrak p}$ is equal to the number of $(-1)$-distinguished
 $K$-orbits of maximal dimension in $\mathfrak p$, and also to the number
of orbits
$({\rm Ad}\,\mathfrak r) \cdot x$ of maximal dimension, where $x$ is a compact element in
$\mathfrak r$.

This reduces the problem to pointing out for every group $G$ from list (ii) of Theorem \ref{list}
the simple real Lie algebra  $\mathfrak s$ corresponding to it, and then to finding the number of orbits
$({\rm Ad}\,\mathfrak s) \cdot x$, where $x$ is a compact element of  $\mathfrak s$,
such that their dimension is maximal.

Now we shall perform this for every group from list  (ii) of Theorem \ref{list},
except those from Remark \ref{rem1} that have already been considered above.

     \vskip 1mm
(4a) Let $G$ be one of the groups $({\sf B}_r, 2\varpi_1)$, $({\sf D}_r, 2\varpi_1)$, $({\sf A}_3, 2\varpi_2)$, $({\sf C}_2, 2\varpi_1)$, $({\sf A}_1, 4\varpi_1)$.\;Therefore,   $\mathfrak k=\mathfrak{so}_n$, where, respectively, $n=2r+1$ (with $r\geqslant 3$), $2r$ (with $r\geqslant 4$), $6$, $5$, $3$.\;Hence the maximal compact subalgebra in  $\mathfrak s$
is $\mathfrak {so}_{n,0}$ (see\;\cite{VO}, \cite{CM}, \cite[Table\;1]{PT}).\;In this case,
$\mathfrak s$ is a real form of the Lie algebra $\mathfrak {sl}_n$
(see Summary Table at the end of \cite{PV} and Tables 7, 9 in Reference Chapter
of \cite{VO}).  It follows from this and Table 8  in Reference Chapter
of  \cite{VO}
that $\mathfrak s={\mathfrak {sl}}_{n}(\mathbb R)$.\;According to \cite[Thm.\;8]{PT}, the number of orbits
 $({\rm Ad}\,\mathfrak s)\cdot x$, where $x$ is a nonzero compact element of $\mathfrak s$, is equal to
$1$ if $n$ is add, and to $2$ if  $n$ is even, and in the case of even $n$ both of these orbits have the same dimension.  Therefore, the nullcone
${\mathcal N}_{G,V}$ is irreducible for odd  $n$ and has exactly two irreducible components
for even $n$.

\vskip 1mm
(4b)  Let $G=({\sf C}_r, \varpi_2)$. Therefore, $\mathfrak k={\mathfrak {sp}}_{2r}$, so the maximal compact
subalgebra in  $\mathfrak s$ is
${\mathfrak {sp}}_{r, 0}$ (see\;\cite{VO}, \cite{CM}, \cite[Table\;1]{PT}).\;In this case,
$\mathfrak s$ is a real form of the Lie algebra $\mathfrak {sl}_{2r}$
(see Summary Table at the end of  \cite{PV} and Tables 7, 9 in Reference Chapter
of \cite{VO}).\;It follows from this and Table 8  in Reference Chapter
of \cite{VO}
that $\mathfrak s={\mathfrak {sl}}_{r}({\mathbb H})$.\;According to \cite[Thm.\;8]{PT}, the number of orbits
 $({\rm Ad}\,\mathfrak s)\cdot x$, where $x$ is a nonzero compact element of $\mathfrak s$, is equal to $1$.\;Therefore, the nullcone
${\mathcal N}_{G,V}$ is irreducible.

\vskip 1mm

(4c) Let $G=({\sf A}_7, \varpi_4)$.\;Then $\mathfrak k={\mathfrak {sl}}_8$, so the maximal compact
subalgebra in  $\mathfrak s$ is
${\mathfrak {su}}_8$ (see\;\cite{VO}, \cite{CM}, \cite[Table\;1]{PT}).\;In this case,
$\mathfrak s$ is a real form of the Lie algebra ${\rm E}_7$
(see Summary Table at the end of \cite{PV} and Tables 7, 9 in Reference Chapter
of \cite{VO}). It follows from this and \cite[Table\;5]{PT} that, using E. Cartan's notation,  $\mathfrak s={\rm E}_{7(7)}$. According to \cite[Table\;12]{PT}, for this $\mathfrak s$, the number of
$(-1)$-distinguished $K$-orbits of maximal dimension $(=63)$ in
${\mathcal N}_{K,\mathfrak p}$ is equal to $2$. Therefore,
the number of irreducible componenets of the nullcone
${\mathcal N}_{G,V}$ is equal to $2$ as well.

\vskip 1mm
(4d) Let $G=({\sf C}_4, \varpi_4)$. Therefore, $\mathfrak k={\mathfrak {sp}}_8$, and hence
the maximal compact
subalgebra in  $\mathfrak s$ is
${\mathfrak {sp}}_{4, 0}$ (see\;\cite{VO}, \cite{CM}, \cite[Table\;1]{PT}).\;In this case,
$\mathfrak s$ is a real form of the Lie algebra ${\rm E}_6$
(see Summary Table at the end of \cite{PV} and Tables 7, 9 in Reference Chapter
of \cite{VO}).\;It follows from this and \cite[Table\;5]{PT} that $\mathfrak s={\rm E}_{6(6)}$.\;According to \cite[Table\;7]{PT}, for this $\mathfrak s$, there is a unique
$(-1)$-distinguished $K$-orbit of maximal dimension $(=36)$ in
${\mathcal N}_{K,\mathfrak p}$.
 Therefore, the nullcone
${\mathcal N}_{G,V}$ is irreducible.

\vskip 1mm
(4e) Let $G=({\sf D}_8, \varpi_8)$. Therefore, $\mathfrak k={\mathfrak {so}}_{16}$, so
the maximal compact
subalgebra in $\mathfrak s$ is
${\mathfrak {so}}_{16, 0}$ (see\;\cite{VO}, \cite{CM}, \cite[Table\;1]{PT}).\;In this case,
$\mathfrak s$ is a real form of the Lie algebra ${\rm E}_8$
(see Summary Table at the end of \cite{PV} and Tables 7, 9 in Reference Chapter
of  \cite{VO}). It follows from this and \cite[Table\;5]{PT} that $\mathfrak s={\rm E}_{8(8)}$.\;According to \cite[Table\;14]{PT}, for this $\mathfrak s$, there is a unique
$(-1)$-distinguished $K$-orbit of maximal dimension $(=129)$ in
${\mathcal N}_{K,\mathfrak p}$.
Hence  the nullcone
${\mathcal N}_{G,V}$ is irreducible.

\vskip 1mm
(4f) Let $G=({\sf F}_4, \varpi_4)$.\;Therefore,  $\mathfrak k={\mathfrak {f}}_{4}$, so
the maximal compact
subalgebra in $\mathfrak s$ is
${\rm{F}}_{4(-52)}$ (see\;\cite[Sect.\;5]{PT}).\;In this case,
$\mathfrak s$ is a real form of the Lie algebra ${\rm E}_8$
(see Summary Table at the end of \cite{PV} and Tables 7, 9 in Reference Chapter
of \cite{VO}).  It follows from this and \cite[Table\;5]{PT} that $\mathfrak s={\rm E}_{6(-26)}$.\;According to \cite[Table\;9]{PT}, for this $\mathfrak s$,
 there is a unique
$(-1)$-distinguished
$K$-orbit of maximal dimension $(=24)$ in ${\mathcal N}_{K,\mathfrak p}$.
Hence  in this case the nullcone
${\mathcal N}_{G,V}$ is irreducible.
\end{proof}

\begin{remark} {\rm In \cite{P2} is obtained an algorithm that employs only elemen\-ta\-ry geometric operations
  (the orthogonal projection of a finite system of points onto a linear subspace and taking its convex hull) and,
  starting from the system of weights of the $G$-module $V$ and the system of roots of the group $G$,
   finds a finite set of linear subspaces $L$ in $V$ such that the irreducible components of maximal dimension of the nullcone
  ${\mathcal N}_{G, V}$ are the varieties $G\cdot L$. In particular, if the property $({\rm ED}_0)$ holds
  (see above the list of properties after formula \eqref{max}), this algorithm describes all the irreducible components of the nullcone ${\mathcal N}_{G, V}$.\;For instance, this is so for every pair $(G, V)$ from Theorem \ref{list}.\;The computer implementation of this algorithm is obtained in \cite{ACP}.}
\end{remark}

\end{document}